\documentclass[12pt]{amsart}
\usepackage{amsmath, amsthm, amstext}
\usepackage{amssymb, array, amsfonts}
\usepackage[english]{babel}
\usepackage{bbding}
\usepackage{hyperref}
\usepackage{float}
\usepackage{graphics}
\usepackage{graphicx}
\usepackage{enumerate}
\numberwithin{equation}{section}

\renewcommand{\l}{\left}
\renewcommand{\r}{\right}
\renewcommand{\mod}{\,\mathrm{mod}\,}

\def \C{\mathbb{C}}

\def \M2{\mathrm{M}_2}
\def \R{\mathbb{R}}

\def \T{\mathbb{T}}
\def \Z{\mathbb{Z}}

\def \sl2r{\mathrm{SL}(2,\R)}

\def \dirint {\int_{\T}^{\oplus}}
\def \dirintd {\int_{\T^d}^{\oplus}}

\newcommand{\beq}{\begin{equation}}
\newcommand{\eeq}{\end{equation}}

\def\wt{\widetilde}

\def\wt{\widetilde}

\newcommand{\eqdef}{\stackrel{\rm def}{=\kern-3.6pt=}}

\theoremstyle{plain}
\newtheorem{theorem}{\bf Theorem}[section]

\newtheorem{prop}[theorem]{\bf Proposition}
\newtheorem{cor}[theorem]{\bf Corollary}

\theoremstyle{definition}

\theoremstyle{remark}
\newtheorem{remark}[theorem]{\bf Remark}

\renewcommand{\ge}{\geqslant}

\renewcommand{\qed}{\vrule height7pt width5pt depth0pt}
\title[$L^2$ reducibility]{$L^2$-reducibility and localization for quasiperiodic operators}
\author[S. Jitomirskaya]{Svetlana Jitomirskaya}
\address{Department of Mathematics,
	University of California,
	Irvine, CA, 92697--3875,
	United States of America}
\email{szhitomi@math.uci.edu}

\author[I. Kachkovskiy]{Ilya Kachkovskiy}
\address{Department of Mathematics,
	University of California,
	Irvine, CA, 92697--3875,
	United States of America}
\email{ikachkov@uci.edu}
\begin{document}
\maketitle
\begin{abstract}
	We give a simple argument that if a quasiperiodic multi-frequency Schr\"odinger
        cocycle is reducible to a constant rotation for almost all
        energies with respect to the density of states measure, then
        the spectrum of the dual operator is purely point for Lebesgue
        almost all values of the ergodic parameter $\theta$. The
        result holds in the $L^2$ setting provided, in addition, that
        the conjugation preserves the fibered rotation
        number. Corollaries include localization for (long-range)  1D analytic
        potentials with dual ac spectrum and Diophantine frequency as
        well as a new result on multidimensional localization.
\end{abstract}
\section{Introduction}
The relation between spectral properties of discrete quasiperiodic
Schr\"odinger operators and their duals, known as Aubry duality, has a long history
starting from \cite{AA}.  For nice enough $v$,  sufficiently regular normalizable
solutions  of $H\psi = E\psi$ with 
\begin{equation}\label{1}
(H(x)\psi)_n=\psi_{n+1}+\psi_{n-1}+v(x+n\alpha)\psi_n
\end{equation}
yield Bloch wave solutions of $\wt H\psi = E\psi$ for
\begin{equation}\label{dual}
(\wt H(\theta)\psi)_m=\sum\limits_{m'\in \Z}\hat v_{m'} \psi_{m-m'}+2\cos 2\pi(\alpha m+\theta)\psi_m.
\end{equation}
where
$$
v(x)=\sum_{k\in\Z}\hat{v}_k e^{2\pi i k x}.
$$
Conversely, nice enough Bloch waves yield normalizable
eigenfunctions.  There was an expectation stemming from \cite{AA} that this should
translate into the duality between pure point and absolutely continuous
spectra for $H$ and $\wt H$. The fact that such correspondence in the
direction  ``absolutely continuous spectrum for $H$ implies pure point
spectrum for $\wt H$'' could be  false even for the nicest of $v$ was first understood in
\cite{last} which was a surprise at the time. Y. Last showed that for
the  almost Mathieu operator  $H_\lambda$ defined
by \eqref{1} with $v(x)=2\lambda \cos 2\pi x$ and
Liouville $\alpha$ there is absolutely continuous component in the
spectrum of $H$ for $\lambda<1,$ while $\wt H$ has purely singular
continuous spectrum. By now it is known that for $\lambda<1$ the
spectrum of $H$ is purely absolutely continuous for all $x,\alpha$,
while for the corresponding $\wt H$ there is an interesting phase
diagram of singular continuous and pure point phases depending on the
arithmetic properties of both $\alpha$ and $x$ \cite{AYZ, jitsim}. It should be noted that the  almost Mathieu family
$\{H_\lambda\}_{\lambda>0}$ is self-dual with $\wt H_\lambda = \lambda H_{1/\lambda}$.


A more refined notion than absolutely continuous spectrum can be given in terms of 
reducibility of the corresponding Schr\"odinger cocycle, see \cite{Eli}. Reducibility, for instance, 
does imply that the generalized eigenfunctions behave as Bloch
waves.

The duality has been explored at many levels, from physics motivated
gauge invariance background \cite{mz}, to operator-theoretic
\cite{GJLS}, to quantitative \cite{AJ}. The main recent thrust, however, has
been in establishing absolute continuity of the spectrum of $H$ \cite{BJ}, reducibility \cite{Puig}, or more subtle property 
called almost reducibility \cite{AJ}, by proving Anderson localization 
(or almost localization) for the dual operator $\wt H.$ This was
motivated by the development of non-perturbative localization methods,
starting with \cite{Lana}, that led to duality-based
reducibility conclusions without the use of  KAM and thus allowing for
much larger ranges of parameters. 


However, recent celebrated results such as \cite{AK1,AFK} provide
methods to establish non-perturbative reducibility 
directly and independently of localization for the dual model. Thus, a
natural question arises: does reducibility of the Schr\"odinger
cocycles imply localization for $\wt H(\theta)$? Since nice
enough reducibility directly yields nicely decaying eigenfunctions for
the dual model, the main question is that of completeness. This has
been a stumbling block in corresponding arguments, e.g. \cite{YQ}. The
question is non-trivial even for the almost Mathieu family, where the
conjectured regime of localization for a.e. $x$ has been
$e^\beta<\lambda$ where $\beta$ is the upper rate of exponential
growth of denominators of continued fractions approximants to $\alpha$
\cite{lanaICM94}, yet direct localization arguments have stumbled upon technical
difficulties for $\lambda<e^{3/2\beta}$, see \cite{liu,tenmart}.  At the
same time, in the recent preprint \cite{AYZ}, the Almost Reducibility Conjecture, recently resolved by Avila
(see \cite{arc} where the $\beta>0$ case that is needed for \cite{AYZ} is presented),
is combined with the technique from \cite{AFK,HY,YQ} 
to establish reducibility for the dual of the entire $e^\beta<\lambda$ region, 
thus leading to the above question of
completeness of the dual solutions. The authors of \cite{AYZ} use
certain quantitative information on the model, 
in particular, the existence of a large collection of eigenfunctions with
prescribed rate of decay to obtain their completeness, which, ultimately, implies arithmetic 
phase transitions for the almost Mathieu operator.

In the present paper, we obtain
an elementary proof of complete localization 
for the dual model under the assumption of $L^2$-reducibility
of the Schr\"odinger cocycle for $H(x)$ for almost all energies with
respect to the density of states measure (see Theorem \ref{main}). 
We do
not make any further assumptions on $H(x)$ or it's  spectra, however, we require the $L^2$ conjugation to preserve the rotation number of the 
Schr\"odinger cocycle (which is always possible if the conjugation is continuous).

Our result implies a simple proof of 
localization as a corollary of dual reducibility for the almost Mathieu operator throughout the entire
$e^\beta<\lambda$ region (first proved in \cite{AYZ}).
Also, if the spectrum of $H(x)$ is purely absolutely continuous, $v$ is analytic, and $\alpha$ is Diophantine, the assumptions of Theorem \ref{main} automatically hold due to \cite{AFK}, which gives a new interpretation of duality in the direction 
from absolutely continuous spectrum to localization, as originally
envisioned. The argument applies to the multi-frequency case, thus
allowing some conclusions on multidimensional localization (see Theorem
\ref{multidim}). It has further implications to the quasiperiodic XY spin chain, to be
explored in a forthcoming paper \cite{IKXY}.

\section{Preliminaries: Schr\"odinger cocycles, rotation number and
  duality}
\noindent Let
$$
(H(x)\psi)_n=\psi_{n+1}+\psi_{n-1}+v(x+n\alpha)\psi_n, 
$$
where $n\in \Z$, $x=(x_1,\ldots,x_d),\alpha=(\alpha_1,\ldots,\alpha_d)\in \T^d$, and $\alpha$ is incommensurate in the sense that $\{1,\alpha_1,\ldots,\alpha_d\}$ are linearly independent over $\mathbb Q$. 
The family $\{H(x)\}_{x\in \T^d} $ is called a $d$-frequency one-dimensional quasiperiodic operator family.
We will identify $\T^d$ with $[0,1)^d$, and continuous functions on $\T^d$ with continuous $\Z^d$-periodic 
functions on $\R^d$. The eigenvalue equation
\beq
\label{eiv_eq}
\psi_{n+1}+\psi_{n-1}+ v(x+n\alpha) \psi_n =E\psi_n
\eeq
can be written in the following form involving transfer matrices,
$$
\begin{pmatrix}\psi_n\\ \psi_{n-1}\end{pmatrix}=\l(\prod_{j=n-1}^{0}S_{v,E}(x+j\alpha)\r)\begin{pmatrix}\psi_{0}\\ \psi_{-1}\end{pmatrix},
$$
where
$$
S_{v,E}(x)=\begin{pmatrix}
E-v(x)&-1\\
1&0
\end{pmatrix},
$$
and the pair $(\alpha,S_{v,E})$ is a Schr\"odinger cocycle understood as a map $(\alpha,S_{v,E}):\T^d\times \C^2\to
\T^d\times \C^2$ given by
$(\alpha,S_{v,E}):(x,w) \mapsto (x+\alpha,S_{v,E}(x) \cdot
w)$. Replacing $S_{v,E}$ with $A \in \sl2r$ gives a definition of an $\sl2r$-cocycle.

Suppose that $A$ is an $\sl2r$-cocycle homotopic to the identity (for
example, a Schr\"odinger cocycle). Then there exist continuous functions $w\colon \T^d\times\T\to \R$ and $u\colon \T^d\times\T\to \R^+$ such that
$$
A(x)\begin{pmatrix}\ \cos 2\pi y\\ \sin 2\pi y\end{pmatrix}
=u(x,y)\begin{pmatrix}\ \cos 2\pi (y+w(x,y))\\ \sin 2\pi (y+w(x,y))\end{pmatrix}.
$$
Let $\mu$ be a probability measure invariant under transformation of $\T^d\times\T$ defined by $(x,y)\mapsto (x+\alpha,y+w(x,y))$ such that it projects to Lebesgue measure over the first coordinate. Then the {\it fibered rotation number} of $(\alpha,A)$ is defined as
$$
\rho(\alpha,A)=\l(\int_{\T^d\times\T} w\,d\mu\r)\,\mod\,\Z.
$$
If $A=S_{v,E}$, then we will denote its rotation number by $\rho(E)$, ignoring the dependence on $\alpha$ and $v$.

If $H(\cdot)$ is  a quasiperiodic operator family, and $E_{H(x)}(\Delta)$ are the
corresponding spectral projections, then the density of states measure 
is defined as
$$
N(\Delta)=\int_{\T^d}(E_{H(x)}(\Delta)\delta_0,\delta_0)\,dx,
$$
$\Delta\subset \R$ is a Borel set. The distribution function of this
measure $N(E)=N((-\infty,E))=N((-\infty,E])$ is called the {\it
  integrated density of states} (IDS). A remarkable relation (see, for
example, \cite{JM,AS} for one-frequency case and
\cite{DelSou} for the general case) is that for Schr\"odinger
cocycles we have $N(E)=1-2\rho(E)$, and $\rho$ is a continuous
function mapping $\Sigma$ onto $[0,1/2]$, where $\Sigma=\sigma(H(x))$
is the spectrum which is known to be independent of $x$. The following
holds \cite{JM,DelSou,H}.
\begin{prop}
\label{rational}
Suppose that $E\notin\Sigma$. Then $2\rho(E)\in \alpha_1 \Z+\alpha_2\Z+\ldots+\alpha_d \Z+\Z$.
\end{prop}

{\it Aubry duality} is a relation between spectral properties of $H(x)$ and the dual Hamiltonian
(\ref{dual}). We will formulate it in the multi-frequency form, with the dual operator (\ref{dual}) 
now being an operator in $l^2(\Z^d)$: 
$$
(\wt H(\theta)\psi)_m=\sum\limits_{m'\in \Z^d}\hat v_{m'} \psi_{m-m'}+2\cos 2\pi(\alpha\cdot m+\theta)\psi_m,
$$
where
$$
v(x)=\sum_{k\in\Z}\hat{v}_k e^{2\pi i k\cdot x}.
$$
Denote the corresponding direct integral spaces (for $H$ and $\wt H$ respectively) by
$$
\mathfrak{H}:=\dirintd l^2(\Z)\,dx,\quad \widetilde{\mathfrak{H}}= {\int_{\T}^{\oplus}l^2(\Z^d)}\,d\theta.
$$
Consider the unitary operator $\mathcal U\colon \mathfrak{H}\to \widetilde{\mathfrak{H}}$ defined on vector functions $\Psi=\Psi(x,n)$ as
$$
(\mathcal {U} \Psi)(\theta,m)=\hat{\Psi}(m,\theta+\alpha\cdot m),
$$
where $\hat\Psi$ denotes the Fourier transform over $x\in \T^d\to m\in \Z^d$ combined with the inverse Fourier transform $n\in \Z\to \theta\in \T$. Let also
$$
\mathcal{H}:=\dirintd H(x)\,dx,\quad \wt{\mathcal{H}}:=\dirint \wt H(\theta)\,d\theta.
$$
Aubry duality can be formulated as the following 
equality of direct integrals.
\beq
\label{Duality}
\mathcal{U} \mathcal H \mathcal U^{-1}=\wt{\mathcal H}.
\eeq 
It is well known (see, for example, \cite{GJLS}) that the spectra of $H(x)$ and $\wt H(\theta)$ coincide for all $x,\theta$. We denote them both by $\Sigma$. 
Moreover, the IDS of the families $H$ and $\wt H$ also
coincide.\footnote{While the proof in \cite{GJLS} is claimed for a more specific
  one-frequency family, it applies without any changes in the mentioned generality.}

A continuous cocycle $(\alpha,A)$ is called $L^2$-{\it reducible} if there exists a matrix function $B\in L^2(\T^d;\mathrm{SL}(2,\R))$ and a (constant) matrix $A_{\star}\in \sl2r$ such that
\beq
\label{reducibility}
B(x+\alpha)A(x)B(x)^{-1}=A_{\star},\quad \text{for a. e.}\quad x\in \T^d.
\eeq
We will call a cocycle $A$ {\it $L^2$-degree $0$ reducible}, if \eqref{reducibility} holds with (possibly complex) $B$ with $|\det B(x)|=1$ and
\beq
\label{deg0}
A_{\star}=\begin{pmatrix}e^{2\pi i \rho(E)}&0\\ 0&e^{-2\pi i \rho(E)} \end{pmatrix}.
\eeq
\section{Obtaining localization by duality}
Our main result is as follows.
\begin{theorem}
	\label{main}
	Suppose that $H(x)$ is a continuous quasiperiodic operator family such that for almost all $E\in \Sigma$ with respect to the density of states measure the cocycle $S_{v,E}$ is $L^2$-degree $0$ reducible. Then the spectra of dual Hamiltonians $\wt H(\theta)$ are purely point for almost all $\theta\in\T$.
\end{theorem}
The following fact is standard and shows that the notion of degree 0 reducibility is needed only in discontinuous setting.
\begin{prop}
	Suppose $(\alpha,A)$ is a continuously reducible cocycle such that \eqref{reducibility} holds with $A_{\star}$ being a rotation matrix. Then it is $L^2$ (and continuously) degree 0 reducible.
\end{prop}
\begin{proof}
	Assume that \eqref{reducibility} holds with $B$ continuous and homotopic to a constant matrix. Since such conjugations $B$ preserve rotation numbers, the matrix $A_{\star}$ will be a rotation by the rotation number of $(\alpha,A)$, and \eqref{deg0} can be obtained by diagonalizing $A_{\star}$ over $\C$ (the diagonalization is another conjugation with a constant complex matrix $B$).
	
	In general, any continuous map from $B\colon \T^d \to \mathrm{SL}(2,\R)$ is homotopic to a rotation $R_{n\cdot x}$ for some $n\in \Z^d$. If $B\colon \T^d\to \mathrm{SL}(2,\R)$ satisfies \eqref{reducibility} and is homotopic to $R_{nx}$, then $R_{-n\cdot x}B(x)$ is homotopic to a constant and also satisfies \eqref{reducibility}.
\end{proof}

\begin{cor}
	Suppose that $v$ is real analytic and $d=1$, $\alpha$ is Diophantine, and 
	that the spectrum of the original operator \eqref{1} is purely absolutely continuous for almost all $x$. Then the spectrum of the dual operator \eqref{dual} is purely point for almost all $\theta$, and 
	the eigenfunctions decay exponentially.
\end{cor}
\begin{proof}
	From \cite[Theorem 1.2]{AFK}, it follows that for Lebesgue almost all $E\in\Sigma$ the cocycle $(\alpha,S_{v,E})$ is rotations reducible. If $\alpha$ is Diophantine, it will also be (analytically) reducible to 	a constant rotation (see, for example, \cite{Puig}). Hence, we fall under the assumptions of 
	Theorem \ref{main}.
\end{proof}
We can also formulate the following (and now elementary) result on multi-dimensional localization.
\begin{theorem}
	\label{multidim}
	If $\alpha$	is Diophantine, then there exists $\lambda_0(\alpha,v)$ such that the long-range lattice operator
	$$
	(\wt H(\theta)\psi)_m=\sum\limits_{m'\in \Z^d}\hat v_{m'} \psi_{m-m'}+2\lambda\cos 2 \pi(\alpha\cdot m+\theta)\psi_m
	$$
	has purely point spectrum for $\lambda>\lambda_0(\alpha,v)$ and almost all $\theta\in \T$.
\end{theorem}
\begin{proof}
	By Eliasson's perturbative reducibility theorem (see \cite{Eli,Amor}), the cocycle $(\alpha,S_{\lambda^{-1} v,E})$ is continuously (even analytically) reducible for $\lambda>\lambda_0(\alpha,v)$ for almost all values of the rotation number. 
	Hence, it satisfies the hypothesis of Theorem \ref{main}.
\end{proof}
\begin{remark}
	We are not aware of other multidimensional localization results with purely arithmetic 
	conditions on the frequency. In \cite{CD}, Theorem \ref{multidim} was proved for general $C^2$ potentials with two critical points (rather than for $\cos x$). Long range operators were also considered. 
	However, the result was obtained for an unspecified full measure set of frequencies. In the book \cite{Bbook}, a possibility of alternative proof was mentioned (for general analytic potentials), with a part of the proof only requiring the Diophantine condition on $\alpha$. 
\end{remark}
\begin{remark}
	For every $v$ and for almost every $E$ among those for which the 
	Lyapunov exponent of the Schr\"odinger cocycle 	$(\alpha,S_{v,E})$ is zero, this cocycle can be $L^2$-conjugated to a cocycle of rotations $R_{\phi(x)}$, see \cite{Kot,DS}, or \cite[Theorem 2.1]{AK1} with a reference to \cite{Si1}. Further conjugation to a constant rotation, such as in the assumption of Theorem \ref{main}, requires solving a cohomological equation $\varphi(x+\alpha)-\varphi(x)=\phi(x)\mod \Z$.
\end{remark}
In the remaining two sections we prove Theorem \ref{main}.
\section{Covariant operator families with many eigenvectors}

\noindent  In this section, we make the following observation: if, for 
a covariant operator family $\{H(\omega)\}$, we have found a single eigenvector of $H(\omega)$ 
for each $\omega$, and the eigenvectors for different $\omega$ on the same 
trajectory are essentially different (that is, not obtained from each other by translation), then, for almost every $\omega$, the eigenvectors of $H(\omega)$ obtained by such translations form a complete set.

Let $(\Omega,d\omega)$ be a measurable space with a Borel probability measure, equipped with a family
$\{T^n,n\in \Z^d\}$ of measure-preserving one-to-one 
transformations of $\Omega$ such that $T^{m+n}=T^m T^n$. Let $\{H(\omega),\omega\in \Omega\}$ be a weakly measurable operator family of bounded operators in $l^2(\Z^d)$, with the property $H(T^n\omega)=T_n H(\omega)T_{-n}$, where $T_n$ is the translation in $\Z^d$ by the vector $n$. We call $\{H(\omega)\}$ a {\it covariant operator family with many eigenvectors} if
\begin{enumerate}
	\item there exists a Borel measurable function $E\colon \Omega\to \R$ such that for almost every $\omega\in \Omega$ the equality $E(T^n \omega)=E(T^m\omega)$ holds only for $m=n$. 
	\footnote{In other words, $E$ should be one-to-one on almost 
		all trajectories. However, the wording ``almost all trajectories'' is ambiguous as there is no natural measure on the set of trajectories.}
	\item There exists a function $u\colon \Omega\to l^2(\Z^d)$, {\it not necessarily measurable}, such that for almost every $\omega\in \Omega$ we have $H(\omega)u(\omega)=E(\omega)u(\omega)$ and $\|u(\omega)\|_{l^2(\Z^d)}=1$.
\end{enumerate}

\begin{theorem}
	\label{main_ergodic}
	Suppose that $\{H(\omega)\}$ is a covariant operator family in $\Z^d$ with many eigenvectors. Then, for almost every $\omega$, the operator $H(\omega)$ has purely point spectrum with eigenvalues $E(T^{-n}\omega)$ and eigenvectors $T_{n}u(T^{-n}\omega)$.
\end{theorem}
\begin{proof}
	Let $E_k(\omega):=E(T^{-k}\omega)$, and $P_k(\omega)$ be the spectral projection of $H(\omega)$ onto the 
	eigenspace corresponding to $E_k(\omega)$. Both these functions are weakly measurable in $\omega$. For almost all $\omega$, we have for all $k\in \Z^d$
	$$
	H(T^k \omega) T_k u(\omega)=T_k H(\omega)u(\omega)=E(\omega)T_k u(\omega)=E_k(T^k\omega) T_k u(\omega).
	$$
	Hence, $T_k u(\omega)$ belongs to the range of $P_k(T^k \omega)$, and, for each basis vector $\delta_l\in l^2(\Z^d)$, we have
	\beq
	\label{inequality}
	(P_k(T^k\omega)\delta_l,\delta_l)\ge |(T_k u(\omega),\delta_l)|^2.
	\eeq
	Let $P(\omega):=\sum_{k\in \Z^d}P_k(\omega)$. By assumption, all $E_k(\omega)$ are different (for almost all $\omega$), and hence $P(\omega)$ is a projection. We have
	\begin{multline}
	\label{equalone}
	\int\limits_{\Omega} (P(\omega)\delta_l,\delta_l)\,d\omega=\int\limits_{\Omega}\sum\limits_{k\in \Z^d}(P_k(\omega)\delta_l,\delta_l)\,d\omega= \\
	=\int\limits_{\Omega} \sum\limits_{k\in \Z^d} (P_k(T^k\omega)\delta_l,\delta_l)\,d\omega \ge 
	\int\limits_{\Omega}\sum_{k\in \Z^d} |(T_k u(\omega),\delta_l)|^2\,d\omega=1.
	\end{multline}
	Since the left hand side is bounded by 1, the last inequality must be an equality, and hence, for any 
	$l$, the inequality \eqref{inequality} must also be an equality for almost every $\omega$. By passing to the countable intersection of these sets of $\omega$, we ultimately obtain a full measure set for which $P(\omega)=I$, and hence the eigenfunctions from the statement of the theorem indeed form a complete set.
\end{proof}
\begin{remark}
	The left hand side of \eqref{equalone} is the ``partial density of states measure'' of $\Omega$ 
	in the sense that it is the density of states obtained only from a part $P(\omega)$ of the spectral 
	measure. The fact that this measure is still equal to 1 implies that, for almost every $\omega$, 
	that part was actually the full spectral measure, which is equivalent to completeness of the eigenfunctions. The result 
	also implies simplicity of the point spectrum and that the function $u(\cdot)$ can actually 
	be replaced by a measurable function with the same properties (because spectral projections are measurable functions, and, due to simplicity, one can choose measurable branches of eigenfunctions; see \cite{GJLS} for details).
\end{remark}

\section{Proof of Theorem \ref{main}}

Let $\Theta_r^+$ be the set of all $\theta\in [0,1/2]$ such that $\rho(E)=\theta$ 
for some $E\in \Sigma$ with $S_{v,E}$ being $L^2$-degree 0 reducible, and $2\theta\notin \alpha_1\Z+\ldots+\alpha_d\Z+\Z$. By the assumptions, this set has full Lebesgue measure in $[0,1/2]$. Moreover, for $\theta\in \Theta_r^+$ there exists exactly one $E$ such that $\rho(E)=\theta$, because $\rho$ is nonincreasing and can be constant only on intervals contained in $\R\setminus \Sigma$, on which the rotation number must be rationally dependent with $\alpha$, see Proposition \ref{rational}. Let us denote this inverse function by $E(\theta)$, initially defined on $\Theta_r^+$, extend it evenly onto $-\Theta_r^+$, and then 1-periodically onto $\R$. Let us now consider
$$
f(x,\theta)=\begin{cases} B_{11}^{-1}(x,E(\theta))/\|B_{11}^{-1}(\cdot,E(\theta))\|_{L^2(\T)},&\theta\in [0,1/2]\\
B_{12}^{-1}(x,E(-\theta))/\|B_{12}^{-1}(\cdot,E(-\theta))\|_{L^2(\T)},&\theta\in [-1/2,0),
\end{cases}
$$
and then also extend $f(x,\theta)$ 1-periodically over $\theta$. Denote the domain in $\theta$ of both functions by $\Theta_r$ which is a similar (even, and then periodic) extension of $\Theta_r^+$. For $\theta\in \Theta_r$, the function $f$ generates Bloch solutions
$$
\psi(x,\theta)_n=e^{2\pi i n \theta}f(x+n\alpha,\theta),\quad n\in \Z,
$$
which, for almost all $x$, are formal solutions of the eigenvalue equation $(H(x)\psi)_n=E(\theta)\psi_n$. It is also easy to see that, if $u(\theta)_m=\hat f(m,\theta)$ is the Fourier transform of $f$ in the variable $x$, then
$$
(\wt H(\theta)u(\theta))_m=E(\theta) u(\theta)_m,\quad m\in \Z^d.
$$
Finally, $E(\theta+k\cdot \alpha)=E(\theta+l\cdot\alpha)$ with $k\neq l$ can only happen if $2\theta\in \Z+\Z\alpha_1+\ldots+\Z\alpha_d$, 
and so, the function $E(\cdot)$ satisfies the assumptions of Theorem \ref{main_ergodic}, from which the result follows.
\qed
\begin{remark}
	For almost all $\theta$, we, in addition, can choose $f(x,-\theta)=\overline{{f(x,\theta)}}$. This is a standard argument (see, e.g.,  \cite{Puig}, \cite{tenmart}, or \cite{Avila_l2} in the discontinuous setting).
	Let
	$$
	F(x,\theta)=\begin{pmatrix}
	f(x,\theta)&\overline{f(x,\theta)}\\
	e^{-2\pi i \theta }f(x-\alpha,\theta) & 	e^{2\pi i \theta }\overline{f(x-\alpha,\theta)}
	\end{pmatrix}.
	$$
	From the equation, it is easy to see that $S_{v,E}F(x,\theta)=F(x+\alpha,\theta)A_{\star}$, and hence 
	$d(x,\theta)=\det F(x,\theta)$ is a function of $x$ invariant under $x\mapsto x+\alpha$. Therefore, it is almost surely a constant function of $x$. 
	If $d(x,\theta)\neq 0$, then we can simply choose $B(x)=F(x,\theta)^{-1}$. Suppose that $d(x,\theta)=0$. Then we have 
	$$
		f(x,\theta)=\varphi(x,\theta)\overline{f(x,\theta)},\quad e^{2\pi i \theta}f(x+\alpha,\theta)=\varphi(x,\theta)e^{-2\pi i \theta}\overline{f(x+\alpha,\theta)}.
	$$
	for some function $\varphi$ with $|\varphi(x,\theta)|=1$ and $\varphi(x+\alpha,\theta)=e^{-4\pi i \theta}\varphi(x,\theta)$. That is only possible if $2\theta$ is a rational multiple of $\alpha$ modulo $\Z$, which excludes at most countable set of $\theta$s.
\end{remark}


\section{Acknowledgements} We would like to thank Qi Zhou for useful
discussions. S.J. is a 2014--15 Simons Fellow. This research was partially
    supported by the NSF DMS--1401204. I.K. was supported by the AMS--Simons Travel Grant 2014--16. 
    We are grateful to the Isaac
    Newton Institute for Mathematical Sciences, Cambridge, for support
    and hospitality during the programme Periodic and Ergodic Spectral
    Problems where this paper was completed.


\begin{thebibliography}{99}
	\bibitem{Amor} Amor~S., {\it H\"older Continuity of the Rotation Number for Quasi-Perioidic Co-Cycles in $SL(2,\R)$}, Comm. Math. Phys. 287 (2009), 565 -- 588.
	\bibitem{AA} Aubry~S. Andr\'e~G., {\it Analyticity breaking and Anderson localization in incommensurate lattices}, Group
	theoretical methods in physics (Proc. Eighth Internat. Colloq., Kiryat Anavim, 1979), pp. 133 -– 164, Ann. Israel Phys. Soc., 3, Hilger, Bristol, 1980.
	\bibitem{arc}Avila~A., {\it Almost reducibility and absolute continuity}, preprint, http://w3.impa.br/~avila/arac.pdf.
	\bibitem{Avila_l2} Avila~A., {\it On point spectrum at critical coupling}, preprint
	\bibitem{AFK}Avila~A., Fayad~B., Krikorian~R., {\it A KAM scheme for $\mathrm{SL}(2,\mathbb R)$ cocycles with Liouvillean frequencies}. Geom. Funct. Anal. 21 (2011), 1001 -- 1019.
	\bibitem{tenmart}Avila~A., Jitomirskaya~S., {\it The Ten Martini problem}, Annals of Math. 170 (2009), no. 1, 303 -- 342.		
	\bibitem{AK1}Avila~A., Krikorian~R., {\it Reducibility or nonuniform hyperbolicity for quasiperiodic Schr\"odinger cocycles}, Annals of Mathematics 164 (2006), 911 -- 940.
	\bibitem{AJ}Avila~A., Jitomirskaya~S. {\it Almost localization and almost reducibility}, Journal of the 
	European Mathematical Society 12 (2010), no. 1, 93 -- 131.
	\bibitem{AYZ}Avila~A., You~J., Zhou~Q., {\it Complete phase transitions for almost Mathieu operator}, preprint.
	\bibitem{AS}Avron~J., Simon~B., {\it Almost periodic Schr\"odinger operators II. The integrated density of states}, Duke Mathematical Journal 50 (1983), no. 1, 369 -- 391.
	\bibitem{Bbook}Bourgain~J., {\it Green's Function Estimates for Lattice Schr\"odinger Operators and Applications}, Annals of Mathematics Studies, Princeton University Press, 2005.
	\bibitem{BJ}Bourgain~J., Jitomirskaya~S., {\it Absolutely continuous spectrum for 1D quasiperiodic operators}, Invent. Math. 148 (2002), 453 -- 463.
	\bibitem{CD}Chulaevsky~V., Dinaburg~E., {\it Methods of KAM-Theory for Long-Range Quasi-Perioidic Operators on $\Z^{\nu}$. Pure Point Spectrum}, Comm. Math. Phys. 153 (1993), 559 -- 577.
	\bibitem{DS}Deift~P., Simon~B., {\it Almost periodic Schr\"odinger operators III. The absolutely continuous spectrum in one dimension}.
	\bibitem{Del}Delyon~F., {\it Absence of localization in the almost Mathieu equation}, J. Phys. A: Math. Gen. 20 (1987), L21 -- L23.
	\bibitem{DelSou}Delyon~F., Souillard~B., {\it The rotation number for finite difference operators and its properties}, Comm. Math. Phys. 89 (1983), no. 3, 415 -- 426.
	\bibitem{Eli}Eliasson~L., {\it Floquet solutions for the 1-dimensional quasi-periodic Schr\"odinger equation}, Comm. Math. Phys. 146 (1992), no. 3, 447 -- 482.
	\bibitem{GJLS}Gordon~A., Jitomirskaya~S., Last~Y., Simon~B., {\it Duality and singular continuous spectrum in the almost Mathieu equation}, Acta Mathematica 178 (1997), 169 -- 183.
	\bibitem{H} Herman, M.-R., {\it Une m\'ethode pour minorer les exposants de Lyapounov et quelques exemples montrant
	le caract\`ere local d'un th\'eor\`eme d'Arnol'd et de Moser sur le tore de dimension $2$}. Comment. Math. Helv. 58 (1983), no. 3, 453 -- 502.
	\bibitem{HY}Hou~X., and You~J, {\it Almost reducibility and non-perturbative reducibility of quasiperiodic linear systems}, Invent. Math 190 (2012), 209 -- 260.
	\bibitem{IKXY}Kachkovskiy~I., {\it On transport properties of quasiperiodic XY spin chains}, preliminary title, in preparation.
	\bibitem{lanaICM94}Jitomirskaya~S., {\it Almost everything about the almost Mathieu operator}, II. XIth International Congress of Mathematical Physics (Paris, 1994), 373 -- 382, Int. Press, Cambridge, MA, 1995.			
	\bibitem{Lana} Jitomirskaya~S., {Metal-Insulator Transitions for Almost Mattieu Operator}, Ann. Math., Vol. 150 (1999), No. 3, 1159--1175.
	\bibitem{jitsim}Jitomirskaya~S., Simon~B., {\it Operators with singular continuous spectrum. III. Almost periodic Schrödinger operators}, Comm. Math. Phys. 165 (1994), no. 1, 201 -- 205.		
	\bibitem{JM}Johnson~R., Moser~J., {\it The rotation number for almost periodic potentials}, Comm. Math. Phys. 84 (1982), 403 -- 438.
	\bibitem{Kot}Kotani~S., {\it Lyaponov indices determine absolutely continuous spectra of stationary random one-dimensional Schr\"odinger operator}, Proc. Kyoto Stoc. Conf., 1982.
	\bibitem{last}Last~Y., {\it A relation between a.c. spectrum of ergodic Jacobi matrices and the spectra of periodic approximants}, Comm. Math. Phys. 151 (1993), 183 -- 192.
	\bibitem{liu}Liu~W., Yuan~X., {Anderson Localization for the Almost Mathieu Operator in Exponential Regime}, Journal of Spectral Theory (to appear), http://arxiv.org/abs/1311.0490.
	\bibitem{mz} Mandelstam~V., Zhitomirskaya~S., {\it 1D-quasiperiodic operators. Latent symmetries}, Comm. Math. Phys. 139 (1991), no. 3, 589 -- 604.
	\bibitem{Puig} Puig~J., {\it A nonperturbative Eliasson’s reducibility theorem}, Nonlinearity 19 (2006), no. 2, 355 -- 376.
	\bibitem{Si1}Simon~B., {\it Kotani theory for one-dimensional stochastic Jacobi matrices}, Comm. Math. Phys. 89 (1983), no. 2, 227 -- 234.
	\bibitem{YQ} You~J., Zhou~Q., {\it Embedding of analytic quasi-periodic cocycles into analytic
	quasi-periodic linear systems and its applications}, Comm. Math. Phys. 323 (2013), 975 -- 1005.
\end{thebibliography}
\end{document}